\documentclass[11pt]{amsart}

\usepackage{amsfonts, amstext, amsmath, amsthm, amscd, amssymb}

\usepackage{graphicx, color}

\usepackage{microtype}

\usepackage[hidelinks,pagebackref,pdftex]{hyperref}

\newtheorem{theorem}{Theorem}[section]
\newtheorem{proposition}[theorem]{Proposition}
\newtheorem{lemma}[theorem]{Lemma}
\newtheorem{corollary}[theorem]{Corollary}

\newtheorem*{namedtheorem}{\theoremname}
\newcommand{\theoremname}{testing}
\newenvironment{named}[1]{\renewcommand{\theoremname}{#1}\begin{namedtheorem}}{\end{namedtheorem}}

\theoremstyle{definition}
\newtheorem{definition}[theorem]{Definition}

\newtheorem{conjecture}[theorem]{Conjecture}
\newtheorem{remark}[theorem]{Remark}
\newtheorem{question}[theorem]{Question}

\title[Unexpected essential surfaces]{Unexpected essential surfaces among exteriors of twisted torus knots}
\author{Thiago de Paiva}
\address[]{School of Mathematics, Monash University, VIC 3800, Australia }
\email[]{thiago.depaivasouza@monash.edu}
 
\begin{document}

\begin{abstract}
The twisted torus knots $T(p, q; r, s)$ are obtained by performing a sequence of $s$ full twists on $r$ adjacent strands of $(p, q)$-torus knots. Morimoto asked whether all twisted torus knots
with essential tori in the exterior fit into one of two families. We prove that the answer to this question is no, by finding two  different new families of toroidal twisted torus knots.
\end{abstract}

\maketitle

\section{Introduction}

Twisted torus knots were introduced by Dean in his doctoral thesis \cite{Thesis}. He was interested in discovering when  Dehn surgeries on knots produce Seifert fibered spaces. Since then, twisted torus knots have become a rich source of research. They are known to be among those knots built with the least amount of tetrahedra \cite{simplest, nextsimplest}.
Their volumes ~\cite{generalizedtwistedtoruslinks}, knot Floer homology \cite{homology}, and bridge spectra ~\cite{Bridge}, and Heegaard splittings~\cite{Heegaard} have been studied.
In this paper, we study essential surfaces in their complements.
            
A \emph{twisted torus knot} is determined by four integers $p, q, r, s$. The knot $T(p, q; r, s)$ is obtained by 
twisting $r$ strands of the $(p, q)$-torus knot a total of $s$ full times.

A nontrivial knot is either \emph{hyperbolic}, a \emph{torus knot}, or a \emph{satellite knot} \cite{Thurston}. We say this is its geometric type or knot type.

There has been a significant effort to classify the geometric types of twisted torus knots. In 2009, Morimoto and Yamada studied essential tori in their exteriors. They showed that when $p \equiv 1$ mod $q$ and $q$ divides $r$, the twisted torus knot $T(p, q; r, s)$ has an essential torus in its complement \cite{Anote}.

This result was generalized by Lee, who proved the following theorem \cite{cable}.

\begin{theorem}[Lee]\label{Lee1}
Let $p, q$ be coprime integers with $1 < q < p$, and let $k$ be an integer
with $1 < kq < p$. Then  $T(p, q; kq, s)$ is the $(q, p+k^2qs)$-cable knot on the torus knot
$T(k, ks+ 1)$ for any integer $s$.
\end{theorem} 

For $| s | \geq 2$, Lee further proved that the knot types of the twisted torus knots $T(p, q; r, s)$ when $r \neq p$ and $r$ is
not a multiple of $q$ are hyperbolic or torus knots (see \cite{LeeTorusknotsobtained}, \cite{hyperbolicity}).
Thus, the geometric types in the case $| s | > 1$ are characterized, but the case $| s | = 1$ remains open. 
 
In addition to the satellite knots of theorem~\ref{Lee1}, Morimoto proved the following theorem \cite{tangle}.

\begin{theorem}[Morimoto]\label{Morimoto1}
Let $e > 0$, $k_1 > 1$, $k_2 > 1$, $x_1 > 0$, $x_2 > 0$ be integers with $gcd(x_1$, $x_2) = 1$. Put
$$p = ((e + 1)(k_1 + k_2 - 1) + 1)x_1 + (e + 1)x_2,$$
$$q = (e(k_1 + k_2 - 1) + 1)x_1 + ex_2,$$
$$r = ((e + 1)(k_1 + k_2 - 1) - k_1 + 2)x_1 + ex_2 \textrm{ and }$$ 
$$s = -1.$$
Then, the twisted torus knot $T(p, q; r, s)$ has an essential torus in the exterior whose companion is the torus
knot $T(k_2, -(e + 1)k_2 - 1)$.

\end{theorem}

Morimoto asked:

\begin{question}\label{questionMorimoto}
Are there twisted torus knots with essential tori which are not as in theorem~\ref{Lee1} or in theorem~\ref{Morimoto1}?

\end{question}

Using the same argument of Morimoto, Lee found another family of composite knots that can be seen as a slight generalization of the composite knots of theorem~\ref{Morimoto1} (when $x_1 = 1$, $k_1' = k_1 + x_2 - 1$, and $k_2' = k_2$) \cite{Composite}.

In this paper, we show that the answer to the question~\ref{questionMorimotoYamada} is certainly \emph{yes}.

Our main results are the followings.

\begin{theorem}\label{answerMorimoto}
For any integer $s> 0$, the twisted torus knot $$T(5 + 4(s - 1), 4; 2, 1)$$ is the $T(2, 5 + 4(s - 1))$-cable knot on the torus knot $$T(2, 3 + 2(s - 1)).$$

\end{theorem}

\begin{theorem}\label{theorem5}
Let $p = ka+2$ and $q = kb+1$ for integers $k>1$, $b>0$ and $a = 2b + 1$. Then, the twisted torus knot $$T(p, q; p - 1, - 1)$$ is the $T(k, p - 1 + kb)$-cable knot on the $K$-knot $$K((a, 2), (a - 1, 1), \dots, (b + 2, 1), (b, 1), (b-1, 1), \dots, (2, 1)).$$ 
\end{theorem}

Morimoto and Yamada realized the difficulty of finding twisted torus knots $T(p, q; r, s)$ which have a closed essential surface in their exterior when $r$ is a prime number. Therefore, they asked \cite{Anote}: 

\begin{question}\label{questionMorimotoYamada}
Does there exist twisted torus knot $T(p, q; r, s)$ containing a closed essential surface in its exterior for a prime number $r > 2$?
\end{question}

For $|s|\geq 2$, the answer to the question~\ref{questionMorimotoYamada} was confirmed to be no; see \cite{hyperbolicity}. Lee showed that essential tori appear only when $r=kq$ is a nontrivial multiple of $q$; see \cite{cable}.

Theorem~\ref{Morimoto1} provides some parameters that answer the question above, as well as the following corollary of theorem~\ref{theorem5}.

\begin{corollary}\label{answerMorimotoYamada}
There are integers $1< q < p$, $s=-1$, and prime numbers $r > 2$ such that the twisted torus knots $T(p, q; r, s)$ in theorem~\ref{theorem5} contain an essential torus in their exteriors. 
\end{corollary}

There has been some confusion in the literature whether twisted torus knots $T(p, q; 2, s)$ can contain essential tori in their exteriors (see theorem 1.1 of \cite{MorimotoEssential}). Theorem~\ref{answerMorimoto} proves such knots do exist.

Our method of proof is to generalise the definition of twisted torus knots to a wider class, and to find
conditions under which these are satellites with implications to twisted torus knots.

\subsection{Acknowledgment}I would like to thank my supervisor, Jessica Purcell, for the guidance and Sangyop Lee for helpful comments and one figure.
 
\section{Generalised twisted torus links}

In this section, we define a family of links, called generalised twisted torus links, given by
the closures of some positive braids, and establish some symmetries between
them in twisted torus knots. 

\begin{definition}Let $r_1, \dots, r_n, s_1, \dots, s_n$ be integers such that $r_1 > \dots > r_n > 1$ and all $s_i > 0$. The generalised twisted torus link with parameters $((r_1, s_1), \dots, (r_n, s_n))$ is defined to be the closure of the following braid
$$(\sigma_1\dots \sigma_{r_1 - 1})^{s_1}(\sigma_1\dots \sigma_{r_2 - 1})^{s_2} \dots (\sigma_1\dots \sigma_{r_n - 1})^{s_n}.$$
Here $\sigma_i$ is a standard generator of the braid group, giving a positive crossing between the $i$-th and $(i+1)$-st strands.
We denote this link by $K((r_1, s_1), (r_2, s_2), \dots, (r_n, s_n))$ and call it  a $K$-link for short.

\end{definition}

$K$-links are a natural extension of twisted torus knots $T(p, q; r, s)$ when $r < p$. 

The twisted torus knot $T(p, q; r, s)$ is equal to the $K$-knot $K((p, q), (r, rs))$ for $s > 0$.

\begin{remark}The family of $K$-links is related to that of $T$-links, introduced by Birman and Kofman to describe Lorenz links ~\cite{newtwis}. However, while $K$-links are defined by starting with the largest $r_1$ and stacking largest to smallest in a braid, $T$-links are defined by stacking smallest to largest. 

Some $K$-links coincide with some $T$-links. $K$-links of the form $$K((r_1, r_1s_1), (r_2, r_2s_2), \dots, (r_n, r_ns_n))$$ coincide with $T$-links of the form $$T((r_n, r_ns_n), \dots, (r_2, r_2s_2),(r_1, r_1s_1)),$$ since both are obtained by the same Dehn fillings on the same boundary tori of a link formed by the torus link $T(r_1, s_1)$ and some unknotted circles encircling some strands of it. Also, if we pull the part $(\sigma_1\dots \sigma_{r_2 - 1})^{s_2}$ of the $T$-link $T((r_2, s_2), (r_1, s_1))$ around the braid closure so that it lies below the part $(\sigma_1\dots \sigma_{r_1 - 1})^{s_1},$ we are able to see that it looks like the $K$-link $K((r_1, s_1), (r_2, s_2))$. In particular, $K$-links and $T$-links agree in the case of twisted torus knots.

$K$-links and $T$-links do not necessarily agree in general for the same parameters. But, the following question comes up naturally.

\begin{question}Are $K$-links and $T$-links the same family of links?

\end{question}

\end{remark}

This following result can be seen as a particular case of theorem 1 of \cite{cable}.

\begin{lemma}\label{toruslemma}
Let $p, q, s$ be integers with $2 \leq q < p$ and $s > 0$. The $K$-link $K((p, q), (q, qs))$ is equal to the torus link $T(q, p + qs)$. 
\end{lemma}

\begin{proof}The $K$-link $K((p, q), (q, qs))$ is obtained from the link $T(p, q)\cup J$, where $J$ is an unknotted circle encircling the first $q$ longitudinal strands of the torus link $T(p, q)$, after doing $(1/s)$-surgery on $J$. But, after deforming the torus link $T(p, q)$ to the torus link $T(q, p)$ by exchanging the meridional and longitudinal strands, the unknotted circle $J$ encircles the first $q$ meridional strands of $T(q, p)$. Now, we can move $J$ just to encircle all the $q$ longitudinal strands of $T(q, p)$. If we do  $(1/s)$-surgery on $J$, we get exactly the torus link $T(q, p + qs)$.
\end{proof}

\begin{proposition} \label{proposition1}
Let $p, q, k, s$ be integers with $2 \leq q < p$ and $k, s > 0$. Then, the $K$-link $K((p, q + k), (q, qs))$ is equal to the $K$-link $K((p + qs, q), (p, k))$.
\end{proposition}

\begin{proof}As in lemma~\ref{toruslemma}, exchanging meridional and longitudinal strands of the part $(\sigma_1 \dots  \sigma_{p-1})^q(\sigma_1 \dots  \sigma_{q-1})^{qs}$ of the $K$-link $K((p, q + k), (q, qs))$ gives $(\sigma_1 \dots  \sigma_{q-1})^{p + qs}$ with the braid $(\sigma_1 \dots  \sigma_{p-1})^k$ stacked in the meridional direction. Isotope to the longitudinal direction, 
giving $K((p + qs, q), (p, k))$.
\end{proof}

The symmetry of the last proposition allows us to know geometric types of $K$-links not obtained by full twists.

\begin{corollary}\label{corollary1}  
Let $p, q, k, s$ be integers with $p, q + k$ coprime, $2 \leq q < p$ and $k> 0$. Then, the $K$-knot $K((p + qs, q), (p, k))$ is hyperbolic for $s > 1$. If $ s = 1,$ then the $K$-knot $K((p + qs, q), (p, k))$ is a torus knot if $k = 1$ and $p = mq + m + 1$ for some integers $m\geq 1$.

\end{corollary}

\begin{proof}If $k > 0$, $q$ can't be a multiple of $q + k.$ Then, from \cite{hyperbolicity}, the $K$-knot $K((p, q + k), (q, qs))$ is hyperbolic if $s > 1$. From the last proposition, the $K$-knot $K((p + qs, q), (p, k))$ is hyperbolic as well. 

If $s = 1$, it is known from \cite{Positively} that $K((p, q + k), (q, qs))$ is a torus knot if $q + k - 1 = q$ and $p = m(q + k) + 1$ for some integers $m\geq 1$. Therefore, from the last proposition,  $K((p + qs, q), (p, k))$   is a torus knot if $k = 1$ and $p = mq + m + 1$.
\end{proof}

We ask more general:

\begin{question}What are the geometric types of the $K$-links in terms of the parameters?

\end{question}

\begin{lemma}\label{lemmaSymmetry}
Let $p - 1 - 2(q - 1) - 1 = k,$ $p - 2 = ka$ and $q - 1 = kb,$ where $a, b, k$ are positive integers with $k > 1.$ The twisted torus knot $T(p, q; p - 1, - 1)$ is equivalent to the $K$-knot 
$$K((p - 1, k), (ka, 1), (ka - 1, 1), \cdots, (k(1 + b) + 1, 1), (kb , 1), \cdots, (2, 1)).$$

\end{lemma}

\begin{proof}Consider $T(p, q; p - 1, -1)$ in the representation as in the left of Figure~\ref{Fig:S6}, where it shows $T(12, 5; 11, -1)$.

Since the $q$ meridional strands in the braid $(\sigma_1\dots \sigma_{p - 1})^{q}$ are in the opposite direction from the $p - 1$ meridional strands in the braid $(\sigma_1\dots \sigma_{p - 2})^{p - 1}$ and the third parameter of $T(p, q; p - 1, -1)$ is the first one minus one, we can push back the first $q - 1$ meridional strands of $(\sigma_1\dots \sigma_{p - 1})^q$ as done for $T(12, 5; 11, -1)$ in the middle and right drawings of Figure~\ref{Fig:S6}. 

After that, we end up with the braid
$$(\sigma_1 \dots \sigma_{q - 1}) \dots (\sigma_{1})(\sigma_1 \dots \sigma_{p - 1})(\sigma_1^{-1})(\sigma_2^{-1}\sigma_1^{-1})\dots$$ $$(\sigma_{q - 3}^{-1} \dots \sigma_1^{-1})(\sigma_{q - 2}^{-1} \dots \sigma_1^{-1})(\sigma_{p - 2}^{-1} \dots \sigma_1^{-1})^{p - 1 - (q - 1)}.$$

\begin{figure}
\includegraphics[scale=0.65]{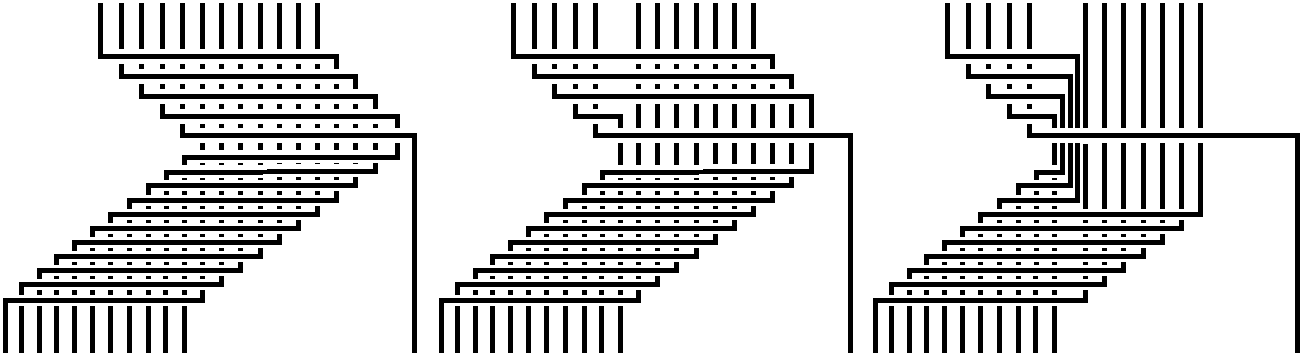}  
  \caption{}
  \label{Fig:S6}
\end{figure}

\begin{figure}
\includegraphics[scale=0.7]{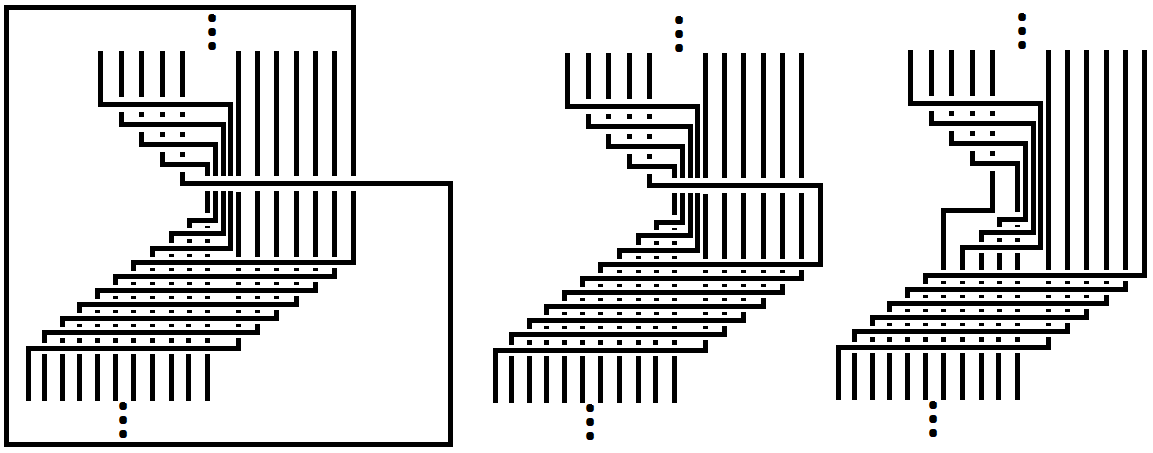}  
  \caption{}
  \label{Fig:S7}
\end{figure}

Apply one Markov move under the furthest right strand to remove it. Then, slide the remaining horizontal strand across the knot, yielding the braid 
$$(\sigma_1 \dots \sigma_{q - 1}) \dots (\sigma_{1})(\sigma_2^{-1})(\sigma_3^{-1}\sigma_2^{-1})\dots$$ $$(\sigma_{q - 2}^{-1} \dots \sigma_2^{-1})(\sigma_{q - 1}^{-1} \dots \sigma_2^{-1})(\sigma_{p - 2}^{-1} \dots \sigma_1^{-1})^{p - 1 - (q - 1) - 1}.$$
This procedure is illustrated in Figure~\ref{Fig:S7} for $T(12, 5; 11, -1)$.

Now, we start to cancel some crossings. In the up part of $T(p, q; p - 1, - 1)$ there are still positive crossings, but we can remove them by sending them to the lower part. We start sending the first crossing around the braid closure to the bottom  so that it gets cancelled as in Figure~\ref{Fig:S9}.
  
\begin{figure}
\includegraphics[scale=0.8]{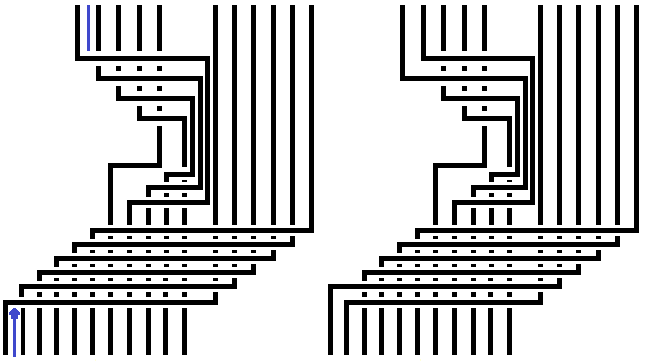} 
  \caption{}
  \label{Fig:S9}
\end{figure}

After doing this to all crossings of the first line, previously given by the braid $(\sigma_1 \dots \sigma_{q - 1})$, it disappears. The lowest line, previously given by the braid $(\sigma_{p - 2}^{-1} \dots \sigma_{1}^{-1})$, becomes $(\sigma_{p - 2}^{-1} \dots \sigma_{q}^{-1})$. Then, after doing the same procedure for the current  highest line, given by the braid $(\sigma_1 \dots \sigma_{q - 2})$, it disappears and the second lowest line, given by the braid $(\sigma_{p - 2}^{-1} \dots \sigma_{1}^{-1})$, becomes $(\sigma_{p - 2}^{-1} \cdots \sigma_{q - 1}^{-1})$. And so we continue. After removing the final positive line, the part given by the braid
$$(\sigma_{p - 2}^{-1} \dots \sigma_{1}^{-1})^{p - 1 - (q - 1) - 1},$$
is replaced by
$$(\sigma_{p - 2}^{-1} \dots \sigma_{1}^{-1})^{p - 1 - (q - 1) - 1 - (q - 1)}(\sigma_{p - 2}^{-1} \dots \sigma_{2}^{-1})(\sigma_{p-2}^{-1} \dots \sigma_{3}^{-1})\dots (\sigma_{p - 2}^{-1} \dots \sigma_{q}^{-1}).$$
Since $p - 1 -2(q - 1) - 1 = k,$ we can rewrite it as 
$$(\sigma_{p - 2}^{-1} \dots \sigma_{1}^{-1})^{k}(\sigma_{p - 2}^{-1} \dots \sigma_{2}^{-1})(\sigma_{p-2}^{-1} \cdots \sigma_{3}^{-1})\dots (\sigma_{p - 2}^{-1} \dots \sigma_{q}^{-1}).$$
In the left drawing of Figure~\ref{Fig:S11} illustrates how $T(12, 5; 11, - 1)$ ends after this step.

\begin{figure}
\includegraphics[scale=0.4]{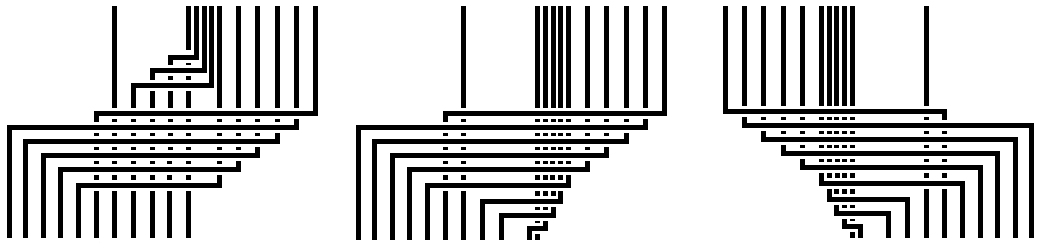}      
  \caption{}
  \label{Fig:S11}
\end{figure}

The part $(\sigma_2^{-1})(\sigma_3^{-1}\sigma_2^{-1})\dots(\sigma_{q - 1}^{-1} \dots \sigma_2^{-1})$ above the last part can be moved down to look like the braid
$$(\sigma_{p - 2}^{-1} \dots \sigma_{p - q + 1)}^{-1})(\sigma_{p - 2}^{-1} \dots \sigma_{p - q + 2}^{-1})\dots (\sigma_{p - 2}^{-1}).$$
In the middle drawing of Figure~\ref{Fig:S11} shows the result of this isotopy for $T(12, 5; 11, - 1).$

Now, the last braid is the mirror image of the following braid (flip the last braid horizontally to see it, as illustrated in the drawing to the right of Figure~\ref{Fig:S11})$$(\sigma_{1} \dots \sigma_{p - 2})^k(\sigma_{1} \dots \sigma_{p - 3}) \dots  (\sigma_{1} \dots \sigma_{p - q - 1})(\sigma_{1} \dots \sigma_{q - 2}) \dots (\sigma_{1}),$$ which its closure is the $K$-knot
$$K((p - 1, k), (p - 2, 1), \dots, (p - q, 1), (q - 1, 1), \dots, (2, 1)).$$

Since $p - 1 - 2(q - 1) - 1 = p - 2q = k$ and $q = kb + 1$, we have $p - q = k(1 + b) + 1.$ Therefore, the twisted torus knot $T(p, q; p - 1, - 1)$ is equivalent to the $K$-knot $$K((p - 1, k), (ka, 1), \dots, (k(1 + b) + 1, 1), (kb, 1), (kb - 1, 1), \dots, (2, 1)). \qedhere $$

\end{proof}

\section{Half Twists}

In this section, we define \emph{half twist}, and using it, we can show one equivalence between $K$-links with different parameters.

\begin{definition}\label{D1} 
Let $L$ be a set of $k$ parallel straight segments on a plane $Y$ which form a trivial braid. Consider $J$ an unknotted circle bounding a disc $D$ which intersects $Y$ transversely in a spanning arc and $L$ in  $k$ points in its interior. A \emph{positive} or \emph{negative} \emph{half twist}  is defined by the following procedure:
cut open $L$ along $D$ and take two copies $D_1$, $D_2$ of $D$, where $D_1$, $D_2$ is the top, bottom disc, respectively. Then, rotate $D_1$ by $\pm (180^{\circ})$ degrees, respectively, and  glue $D_1$ and $D_2$ back together. This entire procedure is illustrated in Figure~\ref{Fig:S14}.

\end{definition}

\begin{figure}
\includegraphics[scale=0.83]{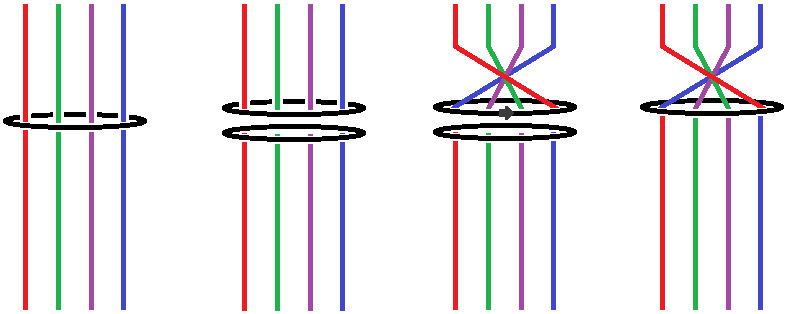}  
  \caption{}
  \label{Fig:S14}
\end{figure}

A positive half twist replaces the $k$ parallel straight segments  with the braid
$$(\sigma_1\sigma_2\dots \sigma_{k - 2}\sigma_{k - 1})(\sigma_1\sigma_2\dots \sigma_{k - 3}\sigma_{k - 2})\dots (\sigma_1\sigma_2)(\sigma_1).$$
On the other hand, a negative half twist transforms the $k$ parallel straight segments into the braid 
$$(\sigma_{k - 1}^{-1}\sigma_{k - 2}^{-1} \dots \sigma_{2}^{-1}\sigma_{1}^{-1})(\sigma_{k - 1}^{-1}\sigma_{k - 2}^{-1} \dots \sigma_{2}^{-1})\dots (\sigma_{k - 1}^{-1}\sigma_{k - 2}^{-1})(\sigma_{k - 1}^{-1}),$$
where $\sigma_{i}^{-1}$ is the inverse of $\sigma_{i}$ in the braid group $B_k$.

Note that a half twist is only defined on a diagram. Half twists on different diagrams may give inequivalent links. See, for example, Figure~\ref{Fig:S18}.

\begin{figure}
\includegraphics[scale=0.6]{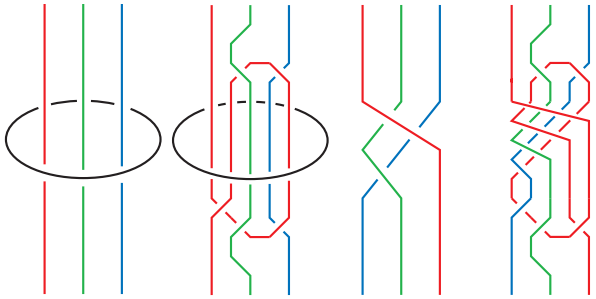}     
  \caption{In the first two drawings, we have different diagrams of the same link. The last two drawings show the effect of  a positive half twist followed by a trivial Dehn filling on the first two diagrams. We see that the results are different.}
  \label{Fig:S18}
\end{figure}

\begin{lemma} \label{lemma3}
The $K$-knot $K((6, 2), (4, 3 + 4(s - 1)))$ is equivalent to the $K$-knot $K((4, 5 + 4(s - 1)), (3, 1))$ for $s > 0.$

\end{lemma}

\begin{proof}Consider the link $L \cup J_1 \cup J_2,$ where $L$ is the link given by the closure of the braid $$(\sigma_1\dots\sigma_5)^{2}(\sigma_1^{-1})(\sigma_3^{-1}),$$
and $J_1, J_2$ is the red, blue circle, respectively, shown in the left of Figure~\ref{Fig:S3}.

By doing $(1/s)$-surgery on $J_1$, we replace $L$ with the closure of the following braid$$(\sigma_1\dots\sigma_5)^{2}(\sigma_1\sigma_2\sigma_3)^{4s}(\sigma_1^{-1})(\sigma_3^{-1}).$$ Next, we do one negative half twist followed by a trivial Dehn filling on $J_2$ replacing the last link by the closure of the following braid $$(\sigma_1\dots\sigma_5)^{2}(\sigma_1\sigma_2\sigma_3)^{3 + 4(s - 1)},$$ which is the $K$-knot $K((6, 2), (4, 3 + 4(s - 1)))$.

\begin{figure}
\includegraphics[scale=0.47]{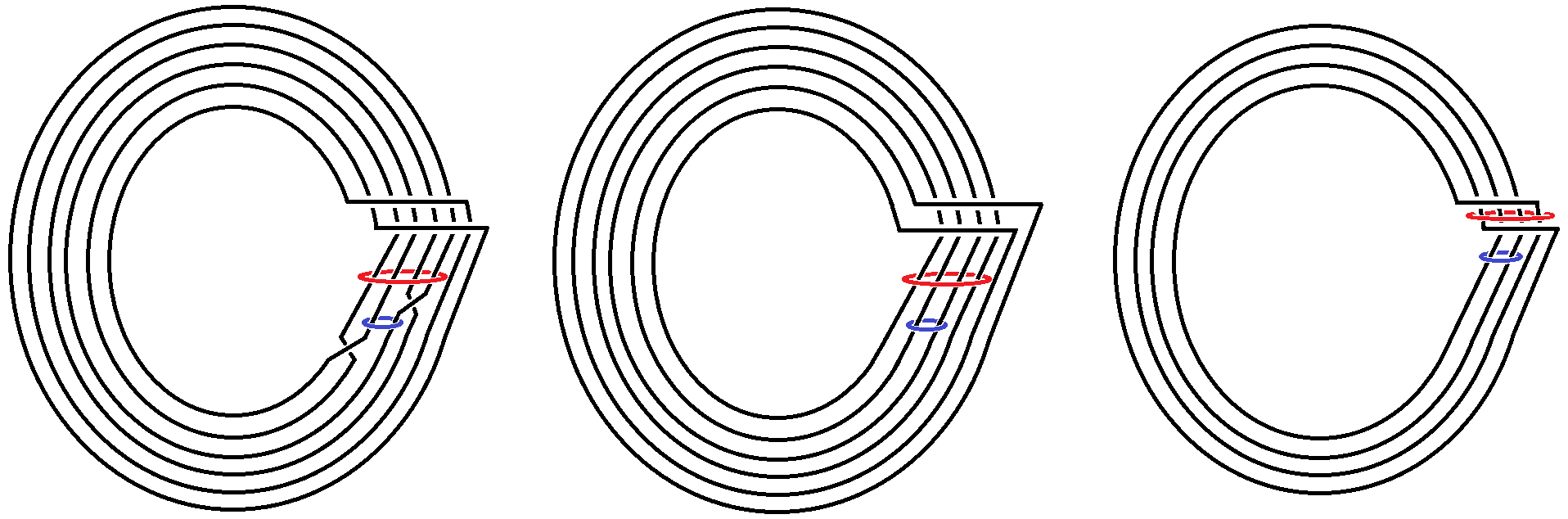} 
  \caption{ }
  \label{Fig:S3}
\end{figure}

The right single crossing of $L$ goes up around the braid closure to be cancelled with the last crossing $\sigma_5$ of $L$. The left single crossing of $L$ goes down around the braid closure to get cancelled with the first crossing $\sigma_1$ of $L$. Result is shown in the middle of Figure~\ref{Fig:S3}. Then, if we push the red circle upwards, we can do some Markov moves to obtain the equivalent link in the right of Figure~\ref{Fig:S3}, where it is formed by the torus link $T(4, 2)$, $J_1$, and $J_2$. 
Now, if we do $(1/s)$-surgery on $J_1,$ we transform $T(4, 2)$ to $T(4, 6 + 4(s - 1)).$ Finally, after doing one negative half twist followed by a trivial Dehn filling on $J_2,$ $T(4, 6 + 4(s - 1))$ is replaced by the $K$-knot $K((4, 5 + 4(s - 1)), (3, 1))$.
\end{proof}

\section{Positive unexpected essential tori}

In this section, we prove theorem~\ref{answerMorimoto}. 

\begin{lemma} \label{lemma4}
The $K$-knot $K((4, 5 + 4(s - 1)), (3, 1))$ is the $T(2, 5 + 4(s - 1))$-cable knot on the torus knot $T(2, 3 + 2(s - 1))$ for $s > 0$.
\end{lemma}

\begin{proof}Start with the case $s = 1$. There is a tube enclosing the first 2 strands of $K((4, 5), (3, 1))$ at the top of its braid. Note if we push this tube following these 2 strands without intersecting them, this tube forms a knotted torus in the exterior of $K((4, 5), (3, 1))$ encircling 2 strands of this $K$-knot. Observe in Figure~\ref{Fig:S4} that this tube forms the torus knot $T(2, 3)$ and the knot inside of it is the torus knot $T(2,5)$.

\begin{figure}
\includegraphics[scale=0.60]{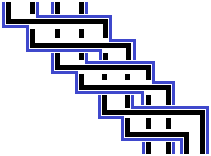} 
  \caption{ }
  \label{Fig:S4}
\end{figure}

If $s  > 1$, we add the braid $(\sigma_1 \sigma_2 \sigma_3)^{4(s - 1)}$ to $K((4, 5), (3, 1))$ above the part $(\sigma_1 \sigma_2 \sigma_3)(\sigma_1 \sigma_2)$ and below the part $(\sigma_1 \sigma_2 \sigma_3)^4$ to get $K((4, 5 + 4(s - 1)), (3, 1))$. Let $T$ be the tube following the same path as the last case. In the part $(\sigma_1 \sigma_2 \sigma_3)^{4(s - 1)}$, $T$ has the form $\sigma_1^{2(s - 1)}$ and the part of $K((4, 5 + 4(s - 1)), (3, 1))$ within of $T$ is given by the braid $\sigma_1^{4(s - 1)}$. Therefore, the knotted torus $T$ is knotted in the form of the torus knot $T(2, 3 + 2(s - 1))$ with the torus knot $T(2, 5 + 4(s - 1))$ inside it. 
\end{proof}

\begin{named}{Theorem~\ref{answerMorimoto}}
For any integer $s> 0$, the twisted torus knot $$T(5 + 4(s - 1), 4; 2, 1)$$ is the $T(2, 5 + 4(s - 1))$-cable knot on the torus knot $$T(2, 3 + 2(s - 1)).$$
\end{named}

\begin{proof}It is known that the twisted torus knot $T(5 + 4(s - 1), 4; 2, 1)$ is equivalent to the twisted torus knot $T(4, 5 + 4(s - 1); 2, 1)$ \cite{Thesis}. From the proposition~\ref{proposition1}, the last twisted torus knot is equivalent to the $K$-knot $K((6, 2), (4, 3 + 4(s - 1)))$, and from lemma~\ref{lemma3} this $K$-knot is the same as the $K$-knot $K((4, 5 + 4(s - 1)), (3, 1))$. Finally, lemma~\ref{lemma4} says the $K$-knot $K((4, 5 + 4(s - 1)), (3, 1))$ is the $T(2, 5 + 4(s - 1))$-cable knot on the torus knot $T(2, 3 + 2(s - 1))$.
\end{proof}

\begin{corollary} 
The $K$-knots $K((5 + 4(s - 1), 4), (2, 2))$ and  $K((4, 5 + 4(s - 1)), (2, 2))$ are the $T(2, 5 + 4(s - 1))$-cable knot on the torus knot $T(2, 3 + 2(s - 1))$ for $s > 0.$
\end{corollary}

\begin{proof}It follows immediately from the last theorem.
\end{proof}

The preprint \cite{MorimotoEssential}(see http://tunnel-knot.sakura.ne.jp/TKSML.pdf) has been cited several times by many authors. Theorem 1.1 of it claims that a twisted torus knot $T(p, q; r, s)$ with parameters $(p, q; 2, s)$ contains no closed essential surfaces in its exterior. However, from the last theorem, all twisted torus knots of the form $T(5 + 4(s - 1), 4; 2, 1)$ with $s > 0$ are cable knots, so each has an essential torus in its exterior. Therefore, they are counterexamples to theorem 1.1 of \cite{MorimotoEssential}. 

After testing a lot of examples of twisted torus knots on SnaPpy, we are able to conjecture the following.

\begin{conjecture} 
Let $r \neq 1$ be not a multiple of $q$. Then, all twisted torus
knots $T(p, q; r, s)$ with $s = 1$ which are satellite knots are the cable knots of theorem~\ref{answerMorimoto}.
\end{conjecture}

If this conjecture is true, we will know all knot types of all twisted torus knots
$T(p, q; r, 1)$.

\section{Negative unexpected essential tori}

In this section, we show theorem~\ref{theorem5}.

Let $K$ be the $K$-knot
 $$K((p - 1, k), (ka, 1), \dots, (k(1 + b) + 1, 1), (kb, 1), \dots, (2, 1))$$
associated to the twisted torus knot $T(p, q; p - 1, - 1)$ in lemma~\ref{lemmaSymmetry}, where $p - 1 - 2(q - 1) - 1 = k,$ $p - 2 = ka$ and $q - 1 = kb$ with $a, b, k$ positive integers and $k > 1$. $K$ can be obtained from the link $T(p - 1, k)\cup J_1 \cup J_2 \cup J_3,$ where $T(p - 1, k)$ is the $(p - 1, k)$-torus link and $J_1, J_2, J_3$ is the green, red, and blue circle, respectively, encircling $ka, k(1 + b), kb$ strands of $T(p - 1, k)$ as in Figure~\ref{Fig:S16}, after doing one positive half twist on $J_1$, then one negative half twist on  $J_2$, followed by one positive half twist on $J_3$ and, finally, trivial Dehn fillings on $J_1, J_2, J_3$.

\begin{figure}
\includegraphics[scale=0.6]{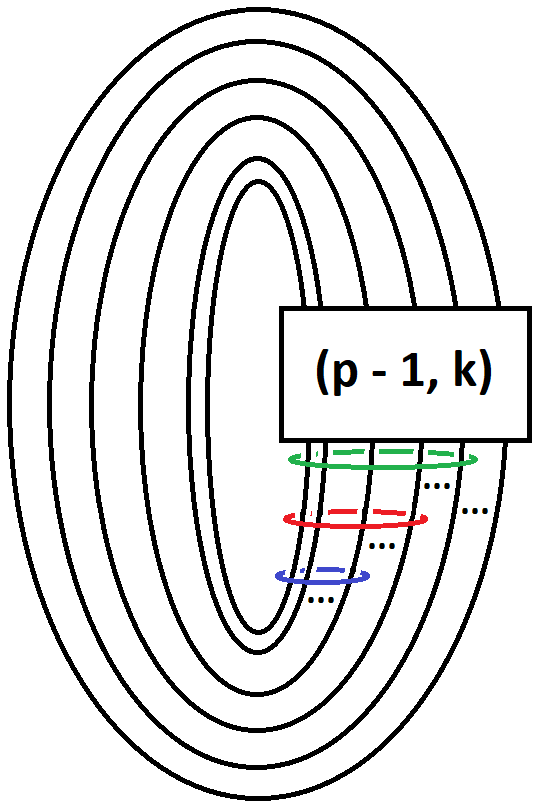} 
  \caption{$J_1, J_2, J_3$ is the green, red and blue circle, respectively.}
  \label{Fig:S16}
\end{figure}

\begin{lemma}\label{lemmawithin}
Suppose the torus link $T(p - 1, k)$ is embedded on the surface of the trivial solid torus $V$. Then, $J_1, J_2, J_3$ can be isotoped to form the trivial torus knots $T(a, 1)$, $T(b + 1, 1)$, $T(b, 1),$ respectively, in $S^3$ to lie in $V$.
\end{lemma}

\begin{proof}We can consider that each circle $J_1, J_2, J_3$ intersects $\partial V$ transversely in just two points, where $\partial V$ is the boundary of $V$. Let's denote $P_1, P_2, P_3$ the rightmost point of $\partial V \cap J_1$, $\partial V \cap J_2$, $\partial V \cap J_3$, respectively, and  $Q_1, Q_2, Q_3$ the leftmost point of $\partial V \cap J_1$, $\partial V \cap J_2$, $\partial V \cap J_3$, respectively. We will pull the points $P_1$, $P_2$, $P_3$ to coincide with the points $Q_1$, $Q_2$, $Q_3$ respectively, after this isopoty.

We start by pulling $P_1$ following the annulus $\partial V - T(p - 1, k)$ in the longitudinal direction clockwise. After passing once around the longitude, $P_1$ is $k$ strands apart from the current rightmost points of $\partial V \cap J_1$. After passing $a - (b + 1)$ times around the longitude, we pull $P_1$ a little bit more to let it to the right of $P_2$. Now, we pull $P_1$ and $P_2$ together following the annulus $\partial V - T(p - 1, k)$ in the longitudinal direction clockwise. After they pass one time around the longitude, we apply one full twist on $J_1$ and $J_2$ on the right direction to keep the linking number of them equal to zero. Then, we pull $P_1$ and $P_2$ a little bit more so that they get on the right of $P_3$. Now, we pull $P_1, P_2, P_3$  together still following the annulus $\partial V - T(p - 1, k)$ in the longitudinal direction clockwise and applying one full twist on the right direction every time they pass around the longitude to let the linking number of them equal to zero. But, since $P_1$ is on the right, $P_2$ is in the middle and $P_3$ is on the left in the final pass around the longitude, we don't twist them in this final pass. We can now connect $P_i$ to $Q_i$.
\begin{figure}
\includegraphics[scale=0.8]{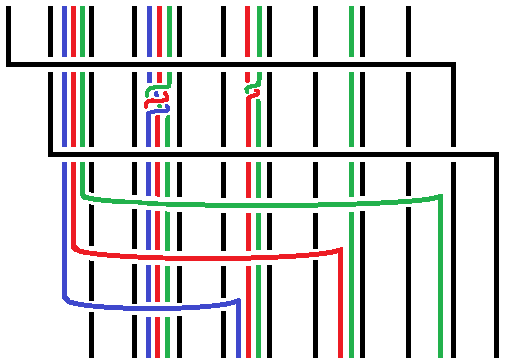} 
  \caption{}
  \label{Fig:S25}
\end{figure}
Since $J_1, J_2, J_3$ passes $a, b + 1, b$ times around the longitude, respectively, they have the forms of the trivial torus knots $T(a, 1),$ $T(b + 1, 1),$ $T(b, 1)$, respectively. To complete the proof, we simply push them simultaneously to the inside of $V$.

Figure~\ref{Fig:S25} illustrates this procedure for when $p = 11, k = 2, a = 4,$ and $b = 2$.
\end{proof}

\begin{lemma}Let $C$ be the trivial torus knot $T(a , 1)$. Consider
unknotted circles $J_1, J_2, J_3$ encircling $a, b + 1, b$, respectively, strands of $C$ as in Figure~\ref{Fig:S17}, where $C, J_1, J_2, J_3$ is  the black, green, red, blue curve, respectively.
Then, there is an isopoty exchanging $C$ and $J_1$ so
that $J_1, J_2, J_3$ get tangled together like in the last lemma.
\end{lemma}

\begin{proof}Since we can write $a = 1a$, $b + 1 = 1(b + 1)$, $b = 1b$, using the same procedure of the proof of the last
lemma, we can deform $J_1, J_2, J_3$ to $T(a, 1), T(b + 1, 1), T(b, 1)$,
respectively, to the interior of $V$, where $V$ is the trivial solid torus which $C$ is on its surface, tangled
together like in the last lemma. To complete the proof, we deform $C$ to the unknot wrapping
around $V$ and isotoped to a meridian of $\partial V$.
\end{proof}

\begin{figure}
\includegraphics[scale=0.75]{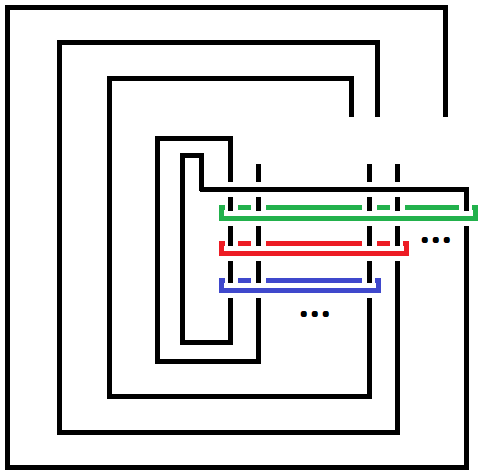} 
  \caption{}
  \label{Fig:S17}
\end{figure}

\begin{lemma} Let $a - 1$ be even of the form $2c$.
The $K$-knot
$$K((p - 1, k), (ka, 1), \dots, (k(1 + b) + 1, 1), (kb, 1), \dots, (2, 1)).$$
associated with the twisted torus knot $T(p, q; p - 1, - 1)$ in lemma~\ref{lemmaSymmetry} is satellite with companion the $K$-knot $$K((a, 2), (a - 1, 1), \dots, (b + 2, 1), (b, 1), (b-1, 1), \dots, (2, 1))$$ and pattern the torus knot $$T(k, p - 1 + kc).$$ Furthermore, $a - 1$ can not be odd.
\end{lemma}

\begin{proof}From lemma~\ref{lemmawithin},  $J_1, J_2, J_3$ can be isotoped to form the trivial torus knots $T(a, 1)$, $T(b + 1, 1)$, $T(b, 1),$ respectively, in $S^3$ so that they lie in $V$, where $V$ is the trivial solid torus that has the $(p - 1, k)$-torus knot $T(p - 1, k)$ on its surface and $J_1, J_2, J_3$ is the green, red, and blue circle, respectively, encircling $ka, k(1 + b), kb$ strands of $T(p - 1, k)$ as shown in Figure~\ref{Fig:S16}.

Consider $C$ an unknotted circle encircling $V$ and isotoped to a meridian of $\partial V$.

We can deform $\partial V$ to the boundary of a tubular neighbourhood of $C$. After this deformation, the torus link $T(p - 1, k)$ becomes the torus link  $T(k, p - 1)$.

By applying the inverse isotopy of the last lemma, $C$ becomes the trivial torus knot $T(a, 1)$
and $J_1, J_2, J_3$ become unknot circles enclosing $a, b + 1, b$ strands, respectively, of
$C$. Therefore, by pushing $T(k, p - 1)$ into $C$, we end with $T(k, p - 1)$ inside an unknotted torus $C$ in the form of the trivial torus knot $T(a, 1)$ and the unknotted circles $J_1, J_2, J_3$ are encircling $a, 1 + b, b$ strands of $T(a, 1),$ respectively, as shown in Figure~\ref{Fig:S17}. 

After doing one positive half twist on $J_1$, $C$ gets the form of the $K$-link $$K((a, 2), (a - 1, 1), \dots, (2, 1)),$$ called $K_1$, and $T(k, p - 1)$ is adjusted by adding the braids 
$$[(\sigma_{1}\dots \sigma_{k - 1})(\sigma_{1}\dots \sigma_{k-2})\dots (2, 1)]^a.$$
Call $B_1$ this new knot.

Then, by doing one negative half twist on $J_2,$ we replace $K_1$ with $$K((a, 2), (a - 1, 1), \dots, (b + 2, 1)),$$ called $K_2$, and add the braids $$[(\sigma_{k - 1}^{-1} \dots \sigma_{1}^{-1})^{-1}( \sigma_{k - 1}^{-1} \dots \sigma_{2}^{-1})\dots (\sigma_{k - 1}^{-1})]^{b+1}$$ to $B_1$, obtaining a new knot called $B_2$.

Finally, we do one positive half twist on $J_3$, replacing $K_2$ with $$K((a, 2), (a - 1, 1), \dots, (b + 2, 1), (b, 1), (b-1, 1), \dots, (\sigma_{1}))$$ and adding the braids $$[(\sigma_{1}\dots \sigma_{k-1})(\sigma_{1}\dots \sigma_{k-2})\dots (\sigma_{1})]^b$$ to $B_2$, getting a new knot called $B_3$. 

Since the negative braids are cancelled with some positive braids, we only need to add the braids $$[(\sigma_{1}\dots \sigma_{k-1})(\sigma_{1}\dots \sigma_{k-2})\dots (\sigma_{1})]^{a - 1}$$ to $T(k, p - 1)$ to get $B_3$ after all these positive or negative half twists. If $a - 1$ is even of the form $2c,$ then  $$[(\sigma_{1}\dots \sigma_{k - 1})(\sigma_{1}\dots \sigma_{k-2})\dots (\sigma_{1})]^{a - 1} = (\sigma_{1}\dots \sigma_{k-1})^{kc},$$ and $B_3$ is the torus knot $T(k, p - 1 + kc).$ 

If $a - 1$ is odd of the form $2c + 1$, then $p = 2+ka = 2 + k(2c+2)$ and $$k = p - 2q = 2+k(2c+2)-2(kb+1) = 2 +2kc+2k-2kb-2 = 2kc+2k-2kb,$$ which implies the contradiction $-1 = 2(c-b)$.\end{proof}

\begin{named}{Theorem~\ref{theorem5}}
Let $p = ka+2$ and $q = kb+1$ for integers $k>1$, $b>0$ and $a = 2b + 1$. Then, the twisted torus knot $$T(p, q; p - 1, - 1)$$ is the $T(k, p - 1 + kb)$-cable knot on the $K$-knot $$K((a, 2), (a - 1, 1), \dots, (b + 2, 1), (b, 1), (b-1, 1), \dots, (2, 1)).$$ 
\end{named}

\begin{proof}By lemma~\ref{lemmaSymmetry}, the twisted torus knot $T(p, q; p - 1, - 1)$ is the $K$-knot
$$K((p - 1, k), (ka, 1), \dots, (k(1 + b) + 1, 1), (kb, 1), \dots, (2, 1)).$$
The last lemma says that this $K$-knot is the $T(k, p - 1 + kb)$-cable knot on the $K$-knot $$K((a, 2), (a - 1, 1), \dots, (b + 2, 1), (b, 1), (b-1, 1), \dots, (2, 1)). \qedhere $$
\end{proof} 

The satellite twisted torus knots $T(p, q; r, s)$ of theorem~\ref{Morimoto1} has $p = ((e + 1)(k_1 + k_2 - 1) + 1)x_1 + (e + 1)x_2$ and 
$r = ((e + 1)(k_1 + k_2 - 1) - k_1 + 2)x_1 + ex_2$. So, $p - r =  (k_1-1)x_1 + x_2.$ Since $k_1 > 1$ and $x_1, x_2 > 0$, thus $p - r > 1.$ We conclude that all satellite twisted torus knots of theorem~\ref{theorem5} are different from those of theorem~\ref{Morimoto1}.

Now, we can answer yes to the question~\ref{questionMorimoto}, since all twisted torus knots of theorem~\ref{answerMorimoto} and theorem~\ref{theorem5} have essential tori in their exteriors because they are all satellites and are not in theorem~\ref{Lee1} or theorem~\ref{Morimoto1}.  
 
Now, we show corollary~\ref{answerMorimotoYamada}.
 
\begin{named}{Corollary~\ref{answerMorimotoYamada}}
There are integers $1< q < p$, $s=-1$, and prime numbers $r > 2$ such that the twisted torus knots $T(p, q; r, s)$ in theorem~\ref{theorem5} contain an essential torus in their exteriors.
\end{named}

\begin{proof}The twisted torus knots $T(8, 3; 7, -1), $ $ T(12, 5; 11, -1),$ $T(20, 9; 19, -1),$ $T(24, 11; 23, -1),$ $T(32, 15; 31, -1)$ are all satellite knots, from the last theorem, where $k = 2$. As a consequence, each has an essential torus in its exterior.
\end{proof}

\bibliographystyle{amsplain}  

\bibliography{Unexpected_essential_surfaces}

\end{document}